\documentclass[a4paper,reqno]{amsart}
\usepackage{amssymb, graphicx, enumerate, enumitem, color}
\usepackage[usenames,dvipsnames,svgnames,table]{xcolor}
\usepackage{textcomp}

\colorlet{mylinkcolor}{violet}
\colorlet{mycitecolor}{YellowOrange}
\colorlet{myurlcolor}{Aquamarine}
\usepackage[pdftitle={Topological minors of cover graphs and dimension},
            pdfauthor={Piotr Micek, Veit Wiechert},
            colorlinks=true,linkcolor=black,citecolor=black,filecolor=black,urlcolor=black
            ]
            {hyperref}

\newtheorem{theorem}{Theorem}

\newtheorem{problem}[theorem]{Problem}

\newtheorem*{claim}{Claim}
\theoremstyle{remark}
\newtheorem{lemma}[theorem]{Lemma}
\newtheorem{obs}[theorem]{Observation}

\newcommand{\set}[1]{\{#1\}}
\newcommand{\norm}[1]{{\left|#1\right|}}
\newcommand{\setR}{\mathbb{R}}

\newcommand{\calX}{\mathcal{X}}
\newcommand{\calC}{\mathcal{C}}
\newcommand{\calD}{\mathcal{D}}
\newcommand{\calE}{\mathcal{E}}

\newcommand{\calH}{\mathcal{H}}
\newcommand{\calF}{\mathcal{F}}

\DeclareMathOperator\Inc{Inc}
\DeclareMathOperator\cover{cover}

\DeclareMathOperator\height{height}
\DeclareMathOperator\val{value}
\DeclareMathOperator\Min{Min}
\DeclareMathOperator\Max{Max}
\DeclareMathOperator\Up{U}
\DeclareMathOperator\D{D}

\let\le\leqslant

\let\leq\leqslant
\let\geq\geqslant

\let\subset\subseteq
\let\subsetneq\varsubsetneq
\let\supset\supseteq

\let\epsilon\varepsilon

\renewenvironment{enumerate}{\begin{enumorig}[label=\textup{(\roman*)}, noitemsep, topsep=2pt plus 2pt, labelindent=.2em, leftmargin=*, widest=iii]}{\end{enumorig}}
\newenvironment{enumeratea}{\begin{enumorig}[label=\textup{(\alph*)}, noitemsep, topsep=2pt plus 2pt, labelindent=.2em, leftmargin=*, widest=iii]}{\end{enumorig}}

\newenvironment{enumerateNested}{\begin{enumorig}[label=\textup{(\alph{enumi}.\arabic*)}, noitemsep, topsep=2pt plus 2pt, labelindent=.2em, leftmargin=*, widest=iii]}{\end{enumorig}}
\newenvironment{enumerateD}{\begin{enumorig}[label=\textup{(d.\arabic*)}, noitemsep, topsep=2pt plus 2pt, labelindent=.2em, leftmargin=*, widest=iii]}{\end{enumorig}}

\renewenvironment{itemize}{\begin{itemorig}[label=\textbullet, noitemsep, topsep=2pt plus 2pt, labelindent=.5em, labelsep=.5em, leftmargin=*]}{\end{itemorig}}

\makeatletter
\let\old@setaddresses\@setaddresses
\def\@setaddresses{\bigskip\bgroup\parindent 0pt\let\scshape\relax\old@setaddresses\egroup}
\makeatother

\begin{document}
\title[Topological minors of cover graphs and dimension]{Topological minors of cover graphs and dimension}

\author[P.~Micek]{Piotr Micek}

\address[P.~Micek]{Theoretical Computer Science Department\\
  Faculty of Mathematics and Computer Science\\
  Jagiellonian University\\
  Krak\'ow\\
  Poland}
  
\email{piotr.micek@tcs.uj.edu.pl}
\thanks{P. Micek is supported by the Mobility Plus program from The Polish Ministry of Science and Higher Education.}

\author[V.~Wiechert]{Veit Wiechert}
\address[V.~Wiechert]{Institut f\"ur Mathematik\\
  Technische Universit\"at Berlin\\
  Berlin \\
  Germany}

\email{wiechert@math.tu-berlin.de}


\thanks{V.\ Wiechert is supported by the Deutsche Forschungsgemeinschaft within the research training group `Methods for Discrete Structures' (GRK 1408).}

\date{\today}




\begin{abstract}
 We show that posets of bounded height whose cover graphs exclude a fixed graph as a topological minor have bounded dimension.
 This result was already proven by Walczak.
 However, our argument is entirely combinatorial and does not rely on structural decomposition theorems.
 Given a poset with large dimension but bounded height, we directly find a large clique subdivision in its cover graph.
 Therefore, our proof is accessible to readers not familiar with topological graph theory, and it allows us to provide explicit upper bounds on the dimension.
 With the introduced tools we show a second result that is supporting a conjectured generalization of the previous result.
 We prove that $(k+k)$-free posets whose cover graphs exclude a fixed graph as a topological minor contain only standard examples of size bounded in terms of $k$.
\end{abstract}

\maketitle

\section{Introduction}

This paper falls into the area of combinatorics of finite partially ordered sets, called \emph{posets}.
The \emph{dimension} of a poset $P$ is the least integer $d$, such that elements of $P$ can be embedded into $\setR^d$ so that $x< y$ in $P$ if and only if 
the point of $x$ is below the point of $y$ with respect to the product order on $\setR^d$.
Equivalently, the dimension of $P$ is the least $d$ such that there are $d$ linear extensions of $P$ whose intersection is $P$. 
This parameter was introduced in~1941 by Dushnik and Miller and is one of the most important measures of a poset's complexity. 
A vast amount of research in the field is concerned with finding reasons or witnesses for high dimension.
And on the other hand, sufficient conditions that give upper bounds for the dimension are of interest.
See Trotter's monograph~\cite{Tro-book} or his chapter in~\cite{Tro-handbook} for a survey on finite posets and dimension theory.

The contribution of this paper is a new approach for upper bounding the dimension of posets. 
We prove two theorems within the same framework.
The first theorem was recently proved by Walczak~\cite{Wal15}. 
The second is new and is a step towards the resolution of questions repeatedly posed in the field.

\begin{theorem}\label{thm:main-thm}
Posets of height at most $h$ whose cover graphs exclude $K_n$ as a topological minor have dimension bounded in terms of $h$ and $n$.
\end{theorem}

There is a long history of research behind this theorem.
In 1977, Trotter and Moore~\cite{TM77} showed that posets whose cover graphs are trees have dimension at most $3$.
More recently, Felsner et al.~\cite{FTW13} showed that posets with outerplanar cover graphs have dimension at most $4$.
One cannot hope for a similar result for posets with planar cover graphs.
Recall that the simplest construction of a $d$-dimensional poset is the \emph{standard example} $S_d$, which is
the poset on $d$ minimal elements $a_1, \dots, a_d$ and $d$ maximal elements $b_1, \dots, b_d$ such that $a_i < b_j$ in $S_d$ if and only if $i \neq j$. 
Already in 1981, Kelly~\cite{Kel81} presented a family of posets with planar cover graphs 
that contain arbitrarily large standard examples as subposets (see Figure~\ref{fig:kelly}), and hence have large dimension.
Note that the height of Kelly's examples grows together with their dimension.
Felsner, Li, and Trotter~\cite{FLT10} proved that posets of height $2$ with planar cover graphs have dimension at most $4$.
Once this result was published, several researchers in the field expressed their believe that posets with planar cover graphs but bounded height should have bounded dimension.
Streib and Trotter showed in~\cite{ST14} that this is indeed the case.
Joret et al.~\cite{JMMTWW} continued this line of research and proved that posets of bounded height whose cover graphs have bounded treewidth also have bounded dimension. 
Note that Theorem~\ref{thm:main-thm} generalizes all these results.

Meanwhile, so after the submission of this manuscript,
Theorem~\ref{thm:main-thm} was significantly generalized.
Together with Gwena\"{e}l Joret we proved that posets with cover graphs in a class of graphs with bounded expansion have dimension bounded by a function of their height~\cite{JMW16}.

We would like to emphasize that the argument in this paper is entirely combinatorial and avoids sophisticated techniques.
In particular, we avoid Ramsey arguments and applications of structural decomposition theorems by Robertson-Seymour~\cite{RS03}, and Grohe-Marx~\cite{GM15}.
Instead, we explicitly construct a subdivision of a large clique in the cover graph of a poset that has large dimension but bounded height.
For that reason, our argument is accessible to readers who are not comfortable with topological graph theory, especially in comparison to the proof in~\cite{Wal15}. 
Moreover, this allows us to give an explicit bound on the dimension in Theorem~\ref{thm:main-thm}.

As witnessed by Kelly's examples, one cannot drop the condition on height in Theorem~\ref{thm:main-thm}.
However, we believe that it can be relaxed.
Bounding the height of a poset is nothing else than forbidding a long chain as a subposet.
A promising line of research is concerned with ($k+k$)-free posets ($k\geq2$), which are defined by excluding two incomparable chains of length $k$ as a subposet.
This class of posets is also a natural generalization of \emph{interval orders}, which are known to be exactly the class of $(2+2)$-free posets.
Over the last few years, a number of nice results~\cite{DJW12,FKT13,LMSTW14} emerged pointing out that problems difficult for the class of all posets might be tractable for $(k+k)$-free posets.
The following question was published in~\cite{Wal15}, but also communicated by a number of other people in the field.
We give support for its positive resolution.
\begin{problem}\label{prob:k+k}
Do $(k+k)$-free posets whose cover graphs exclude $K_n$ as a topological minor have dimension bounded in terms of $k$ and $n$?
\end{problem}
\begin{theorem}\label{thm:standard-ex}
 The $(k+k)$-free posets whose cover graphs exclude $K_n$ as a topological minor contain only standard examples of size bounded in terms of $k$ and $n$.
\end{theorem}
Clearly, the dimension bounds the size of the largest standard example in a poset.
But the converse is not true as interval orders avoid $S_2$ and still can have arbitrarily large dimension~\cite{FHRT92}.
Interestingly, Problem~\ref{prob:k+k} for $k=2$, i.e.~for interval orders, has a positive resolution.
It is a fast corollary from the result of Kierstead and Trotter~\cite{KT00} that
for each interval order $Q$, there is an integer $d$ such that every interval order $P$ with $\dim(P)>d$ contains $Q$ as a subposet.

\begin{figure}[t]
 \centering
 \includegraphics[scale=1.0]{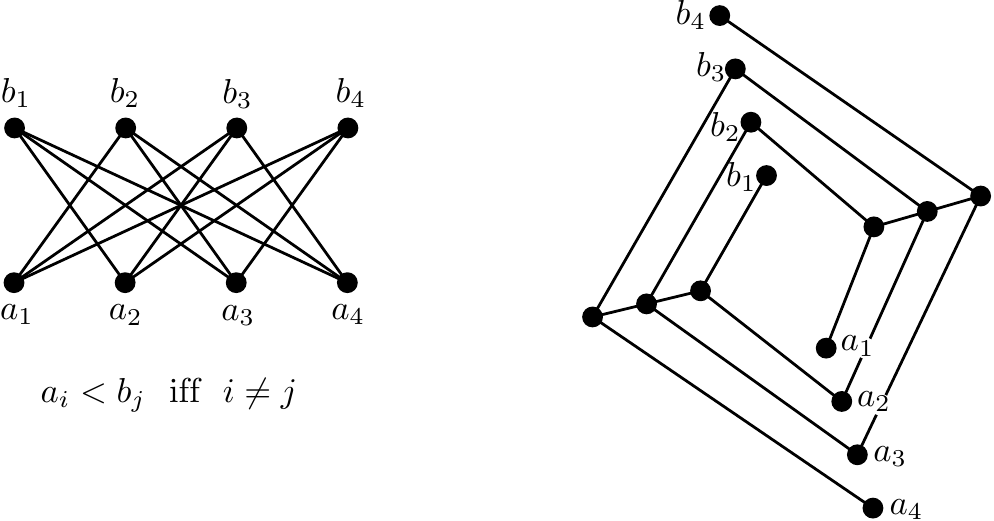}
 \caption{Standard example $S_4$ and Kelly's example containing $S_4$.}
 \label{fig:kelly}
\end{figure}

We finish the introduction with a question asked by several authors~\cite{JMMTWW,Wal15}.
Note that with Theorem~\ref{thm:standard-ex} a positive resolution of Problem~\ref{prob:standard-examples} gives a positive answer for Problem~\ref{prob:k+k} as well.
\begin{problem}\label{prob:standard-examples}
Do posets without $S_d$ as a subposet and whose cover graphs exclude $K_n$ as a topological minor have dimension bounded in terms of $d$ and $n$?
\end{problem}
This problem is wide open already for posets with planar cover graphs.
Trotter and Wang~\cite{TW15+} show that the difference between the dimension and the largest size of a standard example can be arbitrarily large in this case.
However, it is still possible that the dimension of posets with planar cover graphs is linear in the size of their largest standard example.

The paper is organized as follows.
In Section $2$ we introduce basic notations and concepts.
Furthermore, we prove a lemma that plays a key role in our main proof.
In Section $3$ we present a proof of Theorem~\ref{thm:main-thm}, and in Section $4$ Theorem~\ref{thm:standard-ex} is proved.

\section{Preliminaries}\label{sec:preliminaries}

For integers $a,b\geq0$, let $[a]=\set{1,\ldots,a}$ and $[a,b]=\set{a,\ldots,b}$.

By $K_n$ we denote the graph on $n$ vertices with all possible edges.
A \emph{subdivision} of a graph $H$, informally, is any graph $H'$ obtained from $H$ by replacing some edges of $H$ by paths.
Formally, $H'$ contains all vertices of $H$ and for every edge $e$ in $H$ there is a path $P_{e}$ in $H'$ between endpoints of $e$ such that 
the interior of $P_e$ avoids the vertices of $H$, and paths $P_e$, $P_f$ are internally disjoint for all distinct edges $e$, $f$ in $H$.
A graph $H$ is a \emph{topological minor} of $G$ if $G$ contains a subdivision of $H$ as a subgraph.

All posets in the paper are finite.
Elements of a poset $P$ are called \emph{points}.
Points $x,y\in P$ are said to be \emph{comparable} in $P$, if $x\leq y$ or $x\geq y$ in $P$.
Otherwise, $x$ and $y$ are \emph{incomparable} in $P$.
We write $x<y$ in $P$ if it holds that $x\leq y$ and $x\neq y$.
For distinct $x,y\in P$, point $x$ is \emph{covered} by $y$ in $P$ if $x<y$ in $P$ and there is no $z\in P$ with $x<z<y$ in $P$.
In this case, $x<y$ is a \emph{cover relation} of $P$.
The \emph{cover graph} of $P$, denoted by $\cover(P)$, is the graph on the points of $P$ with edges corresponding to cover relations of $P$.
Informally, the cover graph of $P$ can be seen as the undirected graph behind the order diagram of $P$.
A path $x_1,\ldots,x_n$ in $\cover(P)$ is \emph{directed} from $x_1$ to $x_n$, if $x_1<\cdots<x_n$ in $P$.
The \emph{length} of a directed path is the number of its vertices.

A \emph{linear extension} $L$ of $P$ is a poset on the points of $P$ such that the points are pairwise comparable in $L$, and whenever $x \leq y$ in $P$ then $x\leq y$ in $L$.
The \emph{dimension} of $P$, denoted by $\dim(P)$, is the least number $d$ of linear extensions $L_1,\ldots,L_d$ of $P$, such that $x\leq y$ in $P$ if and only if $x\leq y$ in $L_i$ for each $i \in \{1,\ldots,d\}$. 

We let $\Inc(P)=\{(x,y) \mid x, y\in P \textrm{ and } x \textrm{ is incomparable to $y$ in } P\}$ 
denote the set of ordered pairs of incomparable points in $P$.
We say that a point $x\in P$ is \emph{minimal} (\emph{maximal}) if there is no $z\in P$ with $z<x$ in $P$ ($x<z$ in $P$).
We denote by $\Min(P)$ the set of minimal points in $P$ and by $\Max(P)$ the set of maximal points in $P$.
The \emph{downset} of a set $S \subseteq P$ of points is defined as $\D(S)=\set{x\in P\mid \exists s\in S \text{ such that } x\leq s \text{ in }P}$, and similarly we define the \emph{upset} of $S$ to be $\Up(S)=\set{x\in P\mid \exists s\in S\text{ such that }s\leq x\text{ in }P}$. 
For $S=\set{s}$, we write in short $\Up(s)$ and $\D(s)$ instead of $\Up(\set{s})$ and $\D(\set{s})$, respectively.

The \emph{height} of a point $p$ in $P$, denoted by $\height(p)$, is the largest $h$ such that there are $x_1,\ldots,x_h\in P$ with $x_1<\cdots<x_h=p$ in $P$.
Thus, the height of every minimal point in $P$ is $1$.
The \emph{height} of a poset $P$ is the maximum height of its points.
A poset $P$ is $(k+k)$-\emph{free} if it does not contain points $a_1,\ldots,a_k,b_1,\ldots,b_k$ such that $a_1<\cdots<a_k$ in $P$, $b_1<\cdots<b_k$ in $P$, and $a_i$ is incomparable to $b_j$ for each $i,j\in[k]$.

A set $I \subseteq \Inc(P)$ of incomparable pairs is \emph{reversible} if there is a linear extension
$L$ of $P$ with $y<x$ in $L$ for every $(x,y)\in I$. 
Rephrasing the definition of dimension, $\dim(P)$ is the least positive
integer $d$ for which there exists a partition of $\Inc(P)$ into $d$ reversible sets.
An \emph{alternating cycle} in $P$ is a sequence of $r\geq 2$ pairs $(x_1,y_1),\ldots,(x_r,y_r)$ from $\Inc(P)$, such that $x_i\le y_{i+1}$ in $P$ for each $i\in [r]$, where indices are taken cyclically (so we have $x_r\leq y_1$ in $P$).
We will use the following basic fact, originally observed by Trotter and Moore~\cite{TM77} in 1977.
\begin{obs}\label{obs:reversible}
For every poset $P$, a set $I\subset\Inc(P)$ is reversible if and only if $I$ contains no alternating cycle in $P$.
\end{obs}

There is a number of standard observations showing that, in order to bound the dimension,  
we do not need to partition \emph{all} incomparable pairs into reversible sets 
but just a specific subset of these that are in a sense critical for the dimension. 
For our purposes, it is convenient to focus on min-max pairs.
An incomparable pair $(x,y)$ of a poset $P$ is a \emph{min-max pair}, if $x$ is minimal in $P$ and $y$ is maximal in $P$. 
The set of all min-max pairs in $P$ is denoted by $\Inc^*(P)$. 
If $\Inc^*(P)\neq \emptyset$ then define $\dim^*(P)$ as the least positive integer $t$ such that $\Inc^*(P)$ can be partitioned into $t$ reversible sets.
Otherwise, define $\dim^*(P)$ as being equal to $1$. 
The next observation, which is also standard, allows us to work with posets $P$ that have large $\dim^*(P)$.

\begin{obs}\label{obs:min-max-reduction}
For every poset $P$, there is a poset $Q$ such that
\begin{enumerate}
 \item $\height(Q) = \height(P)$,
 \item $\cover(Q)$ can be obtained from $\cover(P)$ by attaching vertices of degree $1$, and
 \item $\dim(P)\leq \dim^*(Q)$.
\end{enumerate}
\end{obs}

The proof idea for Observation~\ref{obs:min-max-reduction} is to build $Q$ by attaching a new minimal point and a new maximal point to every non-extreme point of $P$ (see~\cite{JMTWW} for details).

%

For the presentation of our argument, it is convenient to translate the dimension of a poset into the chromatic number of a certain hypergraph.
For a poset $P$ and a set $I\subset\Inc(P)$, consider the hypergraph $\calH(I)$ with vertex set $I$ and subsets $X\subset I$ forming an edge if the incomparable pairs in $X$ define an alternating cycle in $P$.
Then by Observation~\ref{obs:reversible}, the minimum number of colors needed for a vertex coloring of $\calH(\Inc(P))$ avoiding monochromatic edges is exactly the dimension of $P$. 
For two sets $A,B\subset P$, we define $\Inc(A,B)=\set{(a,b)\in \Inc(P) \mid a\in A,\ b\in B}$ and $\chi(A,B)=\chi(\calH(\Inc(A,B)))$.
With this definition we have $\chi(\Min(P),\Max(P))=\dim^*(P)$ when $\Inc^*(P)\neq\emptyset$.

In a moment we come to the key lemma (Lemma~\ref{lem:unrolling}) for our proof of Theorem~\ref{thm:main-thm}.
The lemma is based on a simple decomposition of minimal and maximal points of $P$, obtained by ``unfolding'' $P$.
This decomposition was first used by Streib and Trotter~\cite{ST14} and was also applied in~\cite{JMTWW,Wal15}.
We turn to its description now.

Suppose $P$ is \emph{connected}, that is, the cover graph of $P$ is connected.
Let $A=\Min(P)$ and $B=\Max(P)$.
Choose arbitrarily $a_0\in A$ and set $A_0=\set{a_0}$.
For $i=1,2,\ldots$ let
 \begin{align*}
  B_i&=\Big\{ b\in B- \bigcup_{1\leq j<i}B_j\mid \text{ there is }a\in A_{i-1}\text{ with }a\leq b\text{ in }P\Big\},\\
  A_i&=\Big\{a\in A- \bigcup_{1\leq j<i}A_j\mid \text{ there is }b\in B_i\text{ with }a\leq b\text{ in }P\Big\}.
 \end{align*}
Let $m$ be the least index with $A_m$ being empty.
Since $P$ is connected, the sets $A_0,\ldots,A_{m-1}$ partition $A$ and the sets $B_1,\ldots,B_m$ partition $B$.
We say that the sequence $A_0,B_1,\ldots,A_{m-1},B_m$ is obtained by \emph{unfolding} $P$ \emph{from} $a_0$.
See an illustration of this decomposition at the top of Figure~\ref{fig:unrolling}.
Also note a useful property of this construction:
\begin{gather}
\begin{align*}
\textrm{for every $a\in A_i$ and $b\in B$ with $a\leq b$ in $P$, we have $b\in B_i \cup B_{i+1}$,}\\
\textrm{for every $b\in B_i$ and $a\in A$ with $a\leq b$ in $P$, we have $a\in A_{i-1} \cup A_{i}$.}
\end{align*}\tag{$\star$}\label{unrolling-property}
\end{gather}
The following lemma intuitively says that each sequence obtained by unfolding a poset contains a ``heavy part'' with respect to dimension.

\begin{figure}[t]
 \centering
 \includegraphics[scale=0.7]{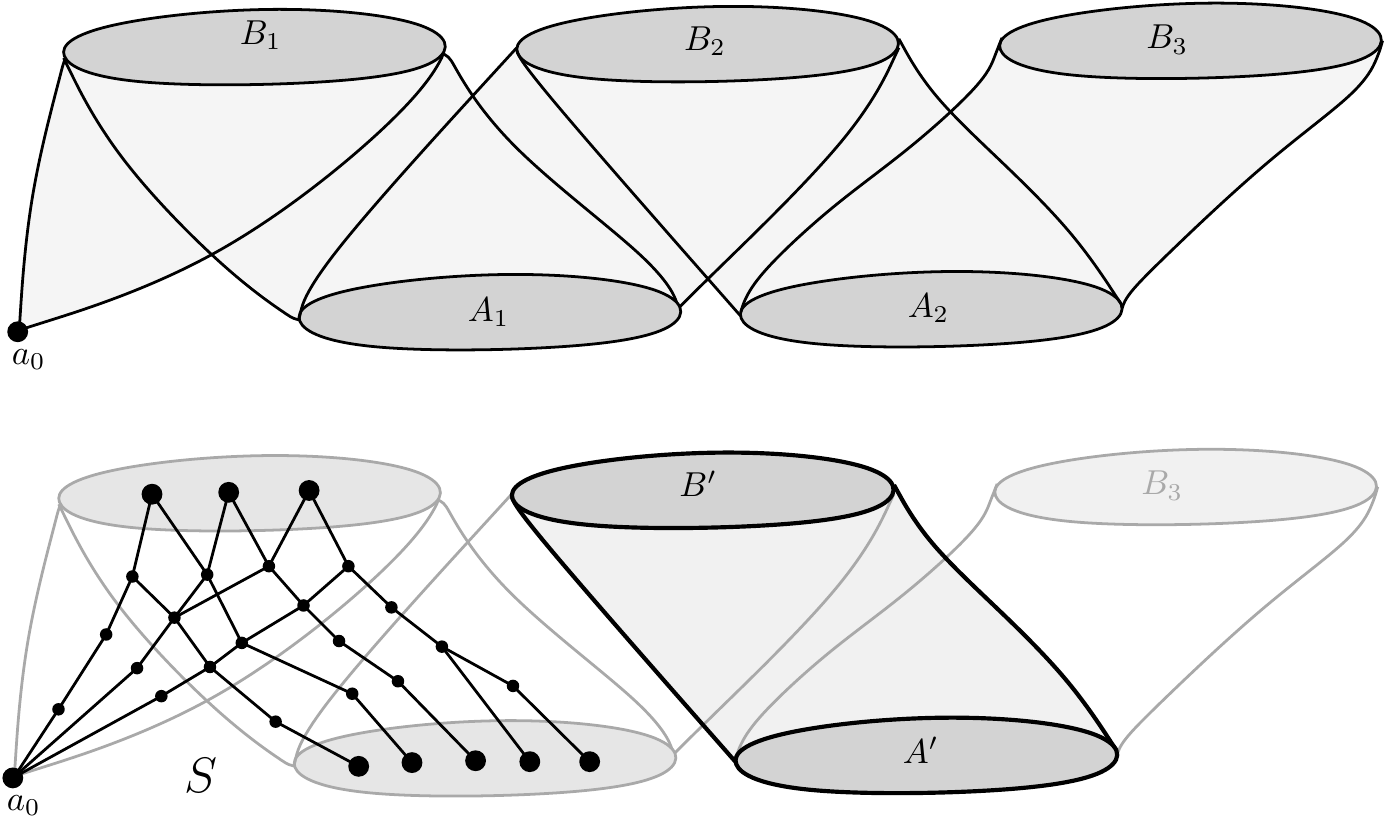}
 \caption{``Unfolding'' the poset $P$ (top figure). Illustration for the definition of $A'$, $B'$, and $S$ for the case $\chi(A_2,B_2)\geq N/2$. (bottom figure)}
 \label{fig:unrolling}
\end{figure}

\begin{lemma}\label{lem:unrolling-chi}
 Let $P$ be a connected poset and $A=\Min(P)$, $B=\Max(P)$ with $\chi(A,B)\geq 3$.
 Consider a sequence $A_0,B_1,\ldots,A_{m-1},B_m$ obtained by unfolding $P$. 
 Then there is $\ell\in[m-1]$ such that
 \[
  \chi(A_{\ell},B_{\ell})\geq \chi(A,B)/2,\quad \textrm{ or }\quad \chi(A_{\ell},B_{\ell+1})\geq \chi(A,B)/2.
 \]
\end{lemma}
\begin{proof}
 Let $\chi=\chi(A,B)$ and $\chi_1=\max_{i\in[1,m-1]}\chi(A_i,B_i)$ as well as $\chi_2=\max_{i\in[1,m-1]}\chi(A_{i},B_{i+1})$.
All we have to show is that
\begin{align}
\chi_1\geq \chi/2\quad\textrm{ or }\quad \chi_2 \geq \chi/2.\label{eq:chi-ineq}
\end{align}
Let $\phi_{i}^{i}$ be a coloring of the hypergraph $\calH(\Inc(A_i,B_i))$ using colors from the set $[\chi_1]$, for $i\in[1,m-1]$.
Let $\phi_{i}^{i+1}$ be a coloring of the hypergraph $\calH(\Inc(A_i,B_{i+1}))$ using colors from the set $[\chi_1+1,\chi_1+\chi_2]$, for $i\in[0,m-1]$.
Combining these colorings we are going to construct a coloring $\phi$ of $\calH(\Inc(A,B))$.
We define two additional distinct colors $c$ and $c'$, and we specify their integer value later on.
We define $\phi$ as follows: for $(a,b)\in\Inc(A,B)$ let
\[
\phi(a,b)=
\begin{cases}
\phi_i^i(a,b) &\mbox{if } a\in A_i,\ b\in B_i, \mbox{ for $i\in[1,m-1]$,}\\
\phi_i^{i+1}(a,b) &\mbox{if } a\in A_i,\ b\in B_{i+1}, \mbox{ for $i\in[1,m-1]$,}\\
c & \mbox{if } a\in A_i,\ b\in B_{j}, \mbox{ for $i,j\in[0,m-1]$ such that $i>j$,}\\
c' & \mbox{if } a\in A_i,\ b\in B_{j}, \mbox{ for $i,j\in[1,m]$ such that $i+1<j$.}\\
\end{cases}
\]

We split our argument in two cases now.
First, we deal with the important case that $\chi_1>0$ and $\chi_2>0$. 
We complete our description of $\phi$ by arbitrarily choosing $c\in [\chi_1]$ and $c'\in[\chi_1+1,\chi_1+\chi_2]$ in this case.
Next, we aim to show that $\phi$ is a proper coloring of $\calH(\Inc(A,B))$.

Suppose for a contradiction that this is not true.
Then there is an alternating cycle $(a_1,b_1),\ldots,(a_r,b_r)$ and a color $c''\in[\chi_1+\chi_2]$, such that $\phi(a_i,b_i)=c''$ for each $i\in[r]$.
We are going to argue for the case $c''\in[\chi_1]$. 
The other case $c''\in[\chi_1+1,\chi_1+\chi_2]$ is symmetric.
First note that for each incomparable pair $(a,b)$ with $a\in A_{i}$, $b\in B_j$, and $\phi(a,b)\in [\chi_1]$, we have $i \geq j$.
Let $a_i\in A_{\ell_i}$ and $b_i\in B_{t_i}$ for $i\in[r]$.
Since $a_i\leq b_{i+1}$ in $P$ (indices are taken cyclically in $[r]$), the \eqref{unrolling-property}-property implies $t_{i+1}\in\set{\ell_i,\ell_i+1}$.
Given $\phi(a_{i+1},b_{i+1})\in[\chi_1]$, we know that $\ell_{i+1}\geq t_{i+1}$.
Therefore, we have $\ell_i\leq t_{i+1} \leq \ell_{i+1}$ and this holds cyclically for $i\in[r]$.
This means that all these values are the same, implying that there is $\ell\in[m-1]$ with $a_i\in A_{\ell}$ and $b_i\in B_{\ell}$ for all $i\in[r]$.
However, the coloring $\phi$ for incomparable pairs in $\Inc(A_\ell,B_\ell)$ agrees with the coloring $\phi_{\ell}^{\ell}$.
Since $\phi_{\ell}^{\ell}$ is proper for $\calH(\Inc(A_\ell,B_\ell))$, our alternating cycle cannot be monochromatic, which is a contradiction.
This proves that $\phi$ indeed yields a proper coloring.
Therefore we have $\chi_1+\chi_2\geq \chi$ and hence~\eqref{eq:chi-ineq} holds in this case.

Now we deal with the case that $\chi_1$ or $\chi_2$ is $0$. (Note that $\chi(X,Y)=0$ if all points from $X$ are below all points from $Y$ in $P$.)
If $\chi_1=\chi_2=0$  then $\phi$ is a proper $2$-coloring of $\calH(\Inc(A,B))$, contradicting $\chi(A,B)\geq 3$.
If $\chi_1>0$ and $\chi_2=0$, then setting $c \in [\chi_1]$ and $c'=\chi_1+1$ we get that $\phi$ is a proper $(\chi_1+1)$-coloring of $\calH(\Inc(A,B))$, 
so $\chi_1 \geq \chi - 1\geq \chi/2$ (as $\chi\geq3$).
The case of $\chi_1=0$ and $\chi_2>0$ goes analogously.
This completes the proof.
\end{proof}

As noted before, Lemma~\ref{lem:unrolling-chi} tells us that when we unfold a poset $P$, then there is ``heavy part'' of it witnessing a large fraction of $P$'s dimension.
In the next lemma we consider the portion of points appearing before this ``heavy part'' and fix a suitable subset of it.
The drawing at the bottom of Figure~\ref{fig:unrolling} will help to understand the statement.

\begin{lemma}\label{lem:unrolling}
For every poset $P$ and sets $A\subset\Min(P)$, $B\subset\Max(P)$ with $\chi(A,B)\geq3$, there are sets $A'\subset A$, $B'\subset B$, and $S\subset \Up(A)\cap \D(B)$ such that
 \begin{enumerate}
  \item\label{item:S-connected} $S$ is connected in $\cover(P)$,
  \item\label{item:chi-inequ} $\chi(A',B')\geq \chi(A,B)/2$,
  \item\label{item:S-support} either $A'\cap\D(S)=\emptyset$, $B'\subset \Up(S)$, or $A'\subset \D(S)$, $B'\cap\Up(S)=\emptyset$.
 \end{enumerate}
\end{lemma}
\begin{proof}
First, using $\chi(A,B)\geq3$ we rule out some trivial cases.
It is well known that the dimension of a poset (not being a union of chains) is witnessed by the dimension of a subposet that is induced by elements contained in a single component of the cover graph.
Translated to our setting, we may assume that there are sets $A''\subset A$ and $B''\subset B$ such that 
$\chi(A'',B'')=\chi(A,B)$, the set $\Up(A'')\cap\D(B'')$ is connected in $\cover(P)$, and $A''\subset \D(B'')$, $B''\subset\Up(A'')$.
Define $Q$ to be the poset that is induced on the points of $\Up(A'')\cap\D(B'')$.

Now choose $a_0\in A''$ arbitrarily and let $A_0,B_1,\ldots,A_{m-1},B_m$ be a sequence obtained by unfolding $Q$ from $a_0$.
By Lemma~\ref{lem:unrolling-chi} there is $\ell\in[m-1]$ such that
\[
 \chi(A_{\ell},B_{\ell})\geq \chi(A'',B'')/2,\quad \textrm{ or }\quad \chi(A_{\ell},B_{\ell+1})\geq \chi(A'',B'')/2\quad\textrm{in }Q.
\]

In the first case we will find a set $S$ fulfilling the first part of item~\ref{item:S-support} and in the second case we will find a set $S$ satisfying the second part of item~(iii).

Suppose first that $\chi(A_{\ell},B_{\ell})\geq \chi(A'',B'')/2$ in $Q$.
Since $Q$ is an induced subposet of $P$, the graph $\mathcal{H}(\Inc(A_\ell,B_\ell))$ with respect to $P$ is equal to $\mathcal{H}(\Inc(A_\ell,B_\ell))$ with respect to $Q$.
Similarly, this holds for $\mathcal{H}(\Inc(A'',B''))$, and hence the inequality $\chi(A_{\ell},B_{\ell})\geq \chi(A'',B'')/2$ also holds with respect to $P$.
Set $A'=A_\ell$ and $B'=B_\ell$.

Now we head for a definition of the set $S$.
By the construction, for every $a\in A_{\ell-1}$ we can fix a path $S(a)$ connecting $a$ and $a_0$ in $\cover(P)$ that is using only points from $\bigcup_{i\in[0,\ell-1]}\Up(A_i)\cap\bigcup_{i\in[1,\ell-1]}\D(B_i)$.
Then we define
\[
 S=\bigcup\nolimits_{a\in A_{\ell-1}}S(a).
\]
See Figure~\ref{fig:unrolling} for an illustration of this definition.
Clearly, $S\subset \Up(A)\cap \D(B)$ and since $a_0$ is contained in $S(a)$ for each $a\in A_{\ell-1}$, the set $S$ is connected in $\cover(P)$.
This proves item~\ref{item:S-connected}.

Now we show that $A'\cap\D(S)=\emptyset$. 
Suppose to the contrary that there is $a'\in A'$ and $s\in S$ with $a \leq s$ in $P$.
Then $s\in S$ implies $s\leq b$ for some $b\in B_i$ with $i\in [1,\ell-1]$, 
and therefore $a' \leq b$ in $P$, which contradicts the \eqref{unrolling-property}-property.
Note also that $B'=B_{\ell} \subset \Up(A_{\ell-1}) \subset \Up(S)$.
This proves item~\ref{item:S-support}.

We are left with the case that $\chi(A_{\ell},B_{\ell+1})\geq \chi(A'',B'')/2$ in $Q$ (and hence in $P$).
We set $A'=A_\ell$ and $B'=B_{\ell+1}$.
Similarly to the previous case, for every $b\in B_{\ell}$ we can fix a path $S(b)$ connecting $b$ and $a_0$ in $\cover(P)$ that is using only points from $\bigcup_{i\in[0,\ell-1]}\Up(A_i)\cap\bigcup_{i\in[1,\ell]}\D(B_i)$.
Then we define
\[
 S=\bigcup\nolimits_{b\in B_{\ell}}S(b).
\]

The properties required for $A'$, $B'$ and $S$ follow along the same lines as for the first case.
This concludes the proof.
\end{proof}

\section{Proof of Theorem~\ref{thm:main-thm}}
Here is an outline of the proof.
Given a poset $P$ with large dimension but bounded height, we are going to construct a subdivision of $K_n$ in the cover graph of $P$.
First, in a preprocessing, we make sure that the dimension of $P$ is witnessed by min-max pairs.
Then we will run through two phases.
In Phase~1 we are going to set up two collections of disjoint sets in $P$ with each set being connected in $\cover(P)$.
We start with two empty collections and then we apply iteratively Lemma~\ref{lem:unrolling} to get new sets for the collections.
At the end of the Phase~1, one of the two collections will be large enough.
This collection is refined in Phase~2, where we fix $n$ points of $P$ that will be the principal vertices of a $K_n$ subdivision in $\cover(P)$.
We conclude Phase~2 with at least $\binom{n}{2}$ sets remaining in our collection.
Finally, we use the sets of the collection to connect the fixed vertices.
This will yield a subdivision of $K_n$.

We omit the trivial cases of the theorem and assume $n\geq3$ and $h\geq2$.
Suppose that $P$ is a poset with $\height(P)\leq h$ and
\begin{align*}
 \dim(P) &> n^L, \textrm{ where } L = 2\cdot\binom{M+h}{h}-1 \textrm{ and } M = \binom{n}{2}^{h^n}.
\end{align*}
First, we apply Observation~\ref{obs:min-max-reduction} to $P$ and obtain a poset $P'$ with $\height(P')\leq h$ and $\dim^*(P')>n^L$.
We are going to find a subdivision of $K_n$ in $\cover(P')$.
Since $\cover(P')$ is obtained by adding extra vertices of degree 1 to $\cover(P)$, this subdivision also exists in $\cover(P)$ (recall that $n\geq 3$).
For convenience, from now on we write $P$ instead of $P'$.

We continue with a description of Phase~1.
During Phase~1 we maintain an additional structure $(A,B,\calC,\calD)$ while running a loop.
After the $i$-th loop iteration we will have the following invariants:
\begin{enumeratea}
 \item $A\subset\Min(P)$, $B\subset\Max(P)$, and $\chi(A,B) > n^{L-i}$,\label{inv:A-B}
 \item $\calC$ is a collection of pairwise disjoint subsets of $P$ with $\norm{\calC}\leq M$ and\label{inv:C}
 \begin{enumerateNested}
  \item $C$ is connected in $\cover(P)$, for every $C\in\calC$,\label{inv:C-connected}
  \item $A \cap \D(C) =\emptyset$ and $B \subset \Up(C)$, for every $C\in\calC$,\label{inv:C-vs-A-and-B}
  \item $\D(c)\cap C=\emptyset$ for each singleton $c$ in $\calC$ and each $C\in\calC -\set{c}$,\label{inv:C-singletons}
 \end{enumerateNested}
 \item $\calD$ is a collection of pairwise disjoint subsets of $P$ with $\norm{\calD}\leq M$ and\label{inv:D}
 \begin{enumerateNested}
  \item $D$ is connected in $\cover(P)$, for every $D\in\calD$,
  \item $A\subset \D(D)$ and $B\cap\Up(D)=\emptyset$, for every $D\in\calD$,
  \item $\Up(d)\cap D=\emptyset$ for each singleton $d$ in $\calD$ and each $D\in\calD -\set{d}$.
 \end{enumerateNested}
\end{enumeratea}

We also have a measure of quality of the maintained structure.
For each $C\in\calC$ with $\norm{C}>1$, we let $\val(C)=h$, and if $C=\set{c}\in \calC$, we let $\val(C)$ be the length of the longest directed path in $\cover(P)$ from $c$ to any $b\in B$.
For each $D\in\calD$ with $\norm{D}>1$, we let $\val(D)=h$, and if $D=\set{d}\in\calD$, we let $\val(D)$ be the length of the longest directed path in $\cover(P)$ from any $a\in A$ to $d$.
Since the height of $P$ is at most $h$, $\val(X)\leq h$ holds for every $X\in \calC\cup\calD$.
Note also that $\val(C)$ and $\val(D)$ depend on the current sets $A$ and $B$ within the structure.
We define the value of a collection $\calX$ of subsets of $P$, denoted by $\val(\calX)$, to be the sequence of size $M$ sorted in a non-decreasing order with one entry $\val(X)$ for each $X\in \calX$, and with $M-\norm{\calX}$ positions filled with '$h+1$' values.
Note that $\val(\calC)$ and $\val(\calD)$ are sequences of length $M$ with sorted values from the set $\set{1,\ldots,h+1}$. 
Therefore, there are at most $\binom{M+h}{h}$ possible values for $\calC$ and $\calD$, respectively.
We say $\val(\calX') < \val(\calX)$, if there is an index $j\in[M]$ such that the first $j-1$ entries of $\val(\calX')$ and $\val(\calX)$ are the same, and the $j$-th entry of $\val(\calX')$ is smaller than the $j$-th entry of $\val(\calX)$.

During Phase 1 the values of the collections $\calC$ and $\calD$ will decrease.
In such a case we say that the quality of our maintained structure is improving.
Intuitively, a small value of $\calC$ is good since then the sets in $\calC$ are somehow close to all points in $B$, which makes it easier to construct a topological minor in the cover graph.

\subsection*{Phase 1: Updating the data structure}
We set up the initial structure as follows: 
$A=\Min(P)$, $B=\Max(P)$, and $\calC$ and $\calD$ are empty. 
Clearly conditions \ref{inv:A-B}-\ref{inv:D} hold for $i=0$.
Note that we start with $\val(\calC)$ and $\val(\calD)$ being the sequence with $M$ entries of '$h+1$'.

Now we run a loop to improve the quality of the data structure.
In each iteration we ask up to three questions about the current structure $(A,B,\calC,\calD)$.
If we get only negative answers, then the loop will terminate and Phase 1 is done.
If we get a positive answer to one of the questions, then we finish the iteration by updating the structure to $(A',B',\calC',\calD')$ that is satisfying conditions \ref{inv:A-B}-\ref{inv:D} and additionally 
 \begin{align*}
\val(\calC')<\val(\calC) &\textrm{ and } \val(\calD')\leq\val(\calD), \textrm{ or }\\
\val(\calC')\leq\val(\calC) &\textrm{ and } \val(\calD')<\val(\calD).
\end{align*}

Since the number of values for the collection $\calC$ (and $\calD$, resp.) is bounded by $\binom{M+h}{h}$, the quality of our structure can be improved at most $L-1=2\binom{M+h}{h}-2$ times so that there will be at most $L$ iterations in total.

Now we are going to describe the $i$-th ($1\leq i \leq L$) iteration in detail.
Let $(A, B,\calC, \calD)$ be the current structure.
Hence it satisfies conditions \ref{inv:A-B}-\ref{inv:D} with respect to $i-1$.
The iteration starts with the evaluation of the following question:
\begin{align}\tag{Q1}\label{q:improving element}
\textrm{Is there a point $p\in P$ such that}
\end{align}
\begin{enumerate}
 \item $\chi(A, B \cap \Up(p)) > n^{L-i}$, and
 \item there is $C\in\calC$ and $c\in C$ such that $c< p$ in $P$?
\end{enumerate}

First, suppose the answer is 'yes' and fix such a point $p\in P$.
In this case we finish the $i$-th iteration by updating the structure to $(A',B',\calC',\calD')$, where
\begin{align*}
A'&=A-\D(p),\ B'=B \cap \Up(p),\\
\calC'&=\calC\cup \set{p}- \set{C\in \calC\mid p\in\Up(C)},\ \textrm{ and }\ \calD'=\calD.
\end{align*}
See Figure~\ref{fig:singleton} for a visualization of the sets $A'$, $B'$.

\begin{figure}[t]
 \centering
 \includegraphics[scale=1.0]{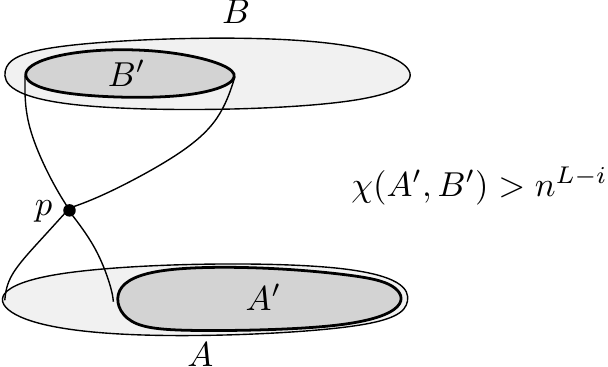}
 \caption{Definition of $A'$ and $B'$ in the case of a 'yes' answer for question~\eqref{q:improving element}.}
 \label{fig:singleton}
\end{figure}

\begin{claim}
 The structure $(A',B',\calC',\calD')$ satisfies the invariants \ref{inv:A-B}-\ref{inv:D}.
 Moreover, $\val(\calC')<\val(\calC)$ and $\val(\calD')\leq\val(\calD)$.
\end{claim}
\begin{proof}
Clearly, $A'\subset A$ is a set of minimal points and $B'\subset B$ is a set of maximal points in $P$. Since the answer for \eqref{q:improving element} is 'yes'
and all points in $A\cap \D(p)$ are below all points in $B\cap \Up(p)$,
we have $\chi(A',B') = \chi(A, B \cap \Up(p)) > n^{L-i}$, so \ref{inv:A-B} holds.

The set $\set{p}$ cannot be contained in $\calC$.
Indeed, otherwise item (ii) of question \eqref{q:improving element} would yield a contradiction to invariant \ref{inv:C-singletons} for $\calC$.
Thus, $\set{p}$ is a new set in $\calC'$ compared to $\calC$.
In order to prove that the sets in $\calC'$ are pairwise disjoint we only need to argue that $p\not\in C$ for every $C\in\calC'-\set{p}$.
But this follows immediately from the definition of $\calC'$.
In order to prove $\norm{\calC'}\leq M$, note that $p$ witnesses the positive answer for \eqref{q:improving element}, so there are $C\in\calC$ and $c\in C$ with $c< p$ in $P$, and therefore $C\not\in\calC'$.
This implies $\norm{\calC'}\leq\norm{\calC}\leq M$.

Item \ref{inv:C-connected} trivially holds.
Item \ref{inv:C-vs-A-and-B} for old sets in $\calC'$ follows immediately from the same invariant for $\calC$, and for $\set{p}$ it follows from the definition of sets $A'$ and $B'$.
For item \ref{inv:C-singletons}, observe first that $\D(p)\cap C=\emptyset$ for every $C\in\calC'-\set{p}$ by the definition of $\calC'$.
We also have to argue that $p\not\in\D(c')$ for each singleton $c'$ in $\calC'-\set{p}\subset\calC$.
Suppose to the contrary that $p\leq c'$ in $P$.
Recall that there is $C\in\calC$ and $c\in C$ with $c<p$.
However, this implies $c<c'$ in $P$, contradicting \ref{inv:C-singletons} for $\calC$.
This completes the verification of \ref{inv:C} for $\calC'$.
Since $\calD'=\calD$ and $A'\subset A$, $B'\subset B$, condition \ref{inv:D} still holds.

Now we show that the quality is improving, that is $\val(\calC')<\val(\calC)$ and $\val(\calD')\leq\val(\calD)$.
The values of sets in $\calC\cap \calC'$ can only decrease as $B'\subseteq B$.
Thus, to show $\val(\calC')<\val(\calC)$ it is enough to argue that $\val(\set{p})$ is smaller than the value of each set removed from $\calC$.
So let $C\in\calC$ such that $p\in\Up(C)$.
If $|C|>1$ then we have $\val(C)=h$ by our definition.
Note that there is no directed path of length $h$ in $\cover(P)$ starting from $p$, since otherwise $p$ must be minimal in $P$, which is not true by item (ii) of question~\eqref{q:improving element}.
Hence $\val(\set{p})\leq h-1<\val(C)$.
If $|C|=1$, that is $C=\set{c}$ for some $c\in P$, then recall that $p$ is not a singleton of $\calC$ and hence $p\neq c$.
Therefore, $p\in\Up(C)$ implies $c<p$ in $P$, which in turn yields $\val(\set{p})<\val(\set{c})$.

Finally, $\val(\calD')\leq\val(\calD)$ holds as $\calD'=\calD$ and $A'\subseteq A$.
This completes the verification of the invariants for the updated structure $(A',B',\calC',\calD')$ in the case of a 'yes' answer to question~\eqref{q:improving element}.
\end{proof}

If the answer for question~\eqref{q:improving element} is 'no', then the procedure continues with a dual question:
\begin{align}\tag{Q2}\label{q:dual improving element}
\textrm{Is there a point $p\in P$ such that}
\end{align}
\begin{enumerate}
 \item $\chi(A\cap\D(p),B) > n^{L-i}$, and
 \item there is $D\in\calD$ and $d\in D$ such that $p < d$ in $P$?
\end{enumerate}
If the answer for~\eqref{q:dual improving element} is 'yes', then we improve the current structure analogously to the 'yes'-case of question~\eqref{q:improving element}. We finish the $i$-th iteration by updating the structure to $(A',B',\calC',\calD')$, where
\begin{align*}
A'&=A\cap\D(p),\ B'=B-\Up(p),\\
\calC'&=\calC,\ \textrm{ and }\ \calD'=\calD\cup \set{p}-\set{D\in \calD\mid p\in\D(D)}.
\end{align*}
The proof that this new structure satisfies conditions \ref{inv:A-B}-\ref{inv:D} and that it improves the quality goes dually to the one for question~\eqref{q:improving element}.

If the answers for questions~\eqref{q:improving element} and~\eqref{q:dual improving element} are both 'no', then the procedure continues with a third question:
\begin{align}\tag{Q3}\label{q:C or D full}
\textrm{Are $\norm{\calC}<M$ and $\norm{\calD}<M$?}
\end{align}

Again, we first deal with the 'yes' answer.
In this case, we are going to show how to find a new candidate set to extend $\calC$ or $\calD$.
We apply Lemma~\ref{lem:unrolling} to sets $A$ and $B$ in $P$ (recall that $\chi(A,B)>n^{L-(i-1)}\geq3$) and get disjoint sets $A'\subseteq A$, $B'\subseteq B$, and $S\subseteq \Up(A)\cap \D(B)$ satisfying
\begin{enumerate}
 \item $S$ is connected in $\cover(P)$,
  \item $\chi(A',B')\geq \chi(A,B)/2$,
 \item either $A'\cap \D(S)=\emptyset$ and $B'\subset \Up(S)$, or $A'\subset \D(S)$ and $B'\cap \Up(S)=\emptyset$.
\end{enumerate}

First, we consider the case that in (iii) we have $A'\cap \D(S)=\emptyset$ and $B'\subset \Up(S)$.
In this case we finish the $i$-th iteration by updating the structure to $(A',B',\calC',\calD')$, where $\calC'=\calC\cup \set{S}$ and $\calD'=\calD$.

\begin{claim}
 The structure $(A',B',\calC',\calD')$ keeps the invariants~\ref{inv:A-B}-\ref{inv:D}.
 Moreover, $\val(\calC')<\val(\calC)$ and $\val(\calD')\leq\val(\calD)$.
\end{claim}
\begin{proof}
Clearly, $A'\subset A$ is a set of minimal and $B'\subset B$ is a set of maximal points in $P$ and
\[
 \chi(A',B')\geq\chi(A,B)/2 >n^{L-(i-1)}/2> n^{L-i}
\]
as $n>2$, 
so~\ref{inv:A-B} holds.

To argue that $\calC'$ is a set of disjoint sets, we need to check whether $S$ is disjoint from every $C\in\calC$.
This holds, since in particular $S\subset \Up(A)$ and on the other hand $\Up(A)\cap C=\emptyset$ (by~\ref{inv:C-vs-A-and-B}), for every $C\in\calC$.
Note also that all sets in $\calC'$ are connected in $\cover(P)$; 
this follows from~\ref{inv:C-connected} for $\calC$ and the fact that $S$ itself is connected in $\cover(P)$.
This proves~\ref{inv:C-connected} for $\calC'$.
Since $A'\subset A$ and $B'\subset B$, invariant~\ref{inv:C-vs-A-and-B} remains true for all $C\in \calC$. 
The new set $S$ has the required property in \ref{inv:C-vs-A-and-B} explicitly.
Therefore, invariant~\ref{inv:C-vs-A-and-B} holds for the whole collection $\calC'$.

For invariant~\ref{inv:C-singletons}, let $c$ be a singleton of $\calC$ and suppose that there is $z\in\D(c)\cap S$.
Since $S\subset \Up(A)\cap \D(B)$, there exists $a\in A$ with $a\leq z$ in $P$.
However, this implies $a\leq c$ in $P$ and hence $\D(c)\cap A\neq \emptyset$, contradicting \ref{inv:C-vs-A-and-B} for $\set{c}\in\calC$.
Therefore, $D(c)\cap S=\emptyset$ for each singleton $c$ in $\calC$.
To complete the verification of \ref{inv:C-singletons} we also have to consider the case where $S$ contains only one point.
So say we have $S=\set{s}$ and suppose to the contrary that $\D(s)\cap C\neq \emptyset$ for some $C\in \calC$.
Let $z\in \D(s)\cap C$ be a point in the intersection.
By disjointness of the sets in $\calC'$, it holds that $z\neq s$ and hence $z<s$ in $P$.
Observe that $s$ fulfills the conditions of question~\eqref{q:improving element}, contradicting the fact that this question was answered with 'no'.
This establishes \ref{inv:C-singletons} for $\calC'$.

Since $\calD'=\calD$, invariant \ref{inv:D} for $\calD'$ is immediate.
Finally, $\val(\calC')<\val(\calC)$ and $\val(\calD')\leq \val(\calD)$ hold as $\calC \subsetneq \calC'$ and $\calD=\calD'$.
\end{proof}

In the second case of (iii), in which we have $A'\subset \D(S)$ and $B'\cap \Up(S)=\emptyset$, 
we finish the $i$-th iteration by updating the structure to $(A',B',\calC',\calD')$, where $\calC'=\calC$ and $\calD'=\calD \cup \{S\}$.
The proof that the invariants are kept in this case goes along similar arguments as in the first case.

Finally, if the answer to all questions~\eqref{q:improving element}-\eqref{q:C or D full} is 'no', then we stop iterating and Phase 1 is done.
\bigskip

Let $(A,B,\calC,\calD)$ be the final data structure of Phase 1 and suppose it is obtained in the $i$-th loop iteration (so $i\in\{0,\ldots,L-1\}$).
Thus, $(A,B,\calC,\calD)$ satisfies the invariants \ref{inv:A-B}-\ref{inv:D} and the answers to questions \eqref{q:improving element}-\eqref{q:C or D full} were 'no' for $(A,B,\calC,\calD)$ in the $(i+1)$-th loop iteration.
The negative answer for \eqref{q:C or D full} says in fact that $\norm{\calC}=M$ or $\norm{\calD}=M$.
Suppose that $\norm{\calC} = M$ from now on (the other case goes dually).

In a moment we will start with Phase 2, which consists of a loop that has $n$ iterations.
In each iteration we find a new principal vertex for the final construction of a $K_n$ subdivision.
Simultaneously, we refine the collection $\calC$ maintaining a large enough subcollection that interacts well with vertices already fixed.

Let us go more into detail now.
It will be convenient to use the following definition.
For a family of sets $\calF$ in $P$ and a point $p\in P$, we define
\[
\calF^p=\set{F\in\calF \mid p\in\Up(F)}. 
\]
While running the loop of Phase 2, we maintain as an invariant a pair $(V,\calE)$ with $V\subset P$ and $\calE\subset\calC$, that is satisfying the following items after the $j$-th loop iteration:
\begin{enumerateD}
 \item $\norm{V}=j$ and $|\calE|\geq M^{(1/h)^{j}} = \binom{n}{2}^{h^{n-j}}$,\label{inv:V-and-calC-sizes}
 \item $V$ is disjoint from every $C\in\calE$, and\label{inv:V-disjoint-from-calC}
 \item for every $v\in V$ and $C\in\calE$, there is $x\in P$ such that $x$ is covered by $v$ in $P$ and $\calE^x=\set{C}$.\label{inv:clean-branching-for-V}
\end{enumerateD}

\subsection*{Phase 2: Selecting the principal vertices}

Before the first iteration we set up the pair $(V,\calE)$ with $V=\emptyset$ and $\calE=\calC$.
Invariants \ref{inv:V-and-calC-sizes}-\ref{inv:clean-branching-for-V} are satisfied for $j=0$ vacuously.

Now we describe the $j$-th iteration of the loop ($1\leq j \leq n$).
Let $(V,\calE)$ be the pair satisfying the invariants after the $(j-1)$-th iteration.
The main issue is to find a new vertex to put into $V$.
We start to look for it from an appropriate vertex in $B$.
We want to pick any vertex from $B-\bigcup_{v\in V}\Up(v)$, so we need to argue that this set is non-empty.
By invariant~\ref{inv:clean-branching-for-V}, we get in particular that for every $v\in V$ there is $C\in\calC$ and $c\in C$ such that $c<v$ in $P$.
Since the answer to question \eqref{q:improving element} was 'no' in Phase~1, 
we have $\chi(A, B\cap \Up(v)) \leq n^{L-(i+1)}$ for every $v\in V$.
Thus,
\begin{align*}
\chi\Big(A,B - \bigcup_{v\in V} \Up(v)\Big) &\geq \chi(A,B) - \sum_{v\in V} \chi(A,B\cap \Up(v))\\
&> n^{L-i} - \norm{V}\cdot n^{L-(i+1)}\\
&> n^{L-i} - n\cdot n^{L-(i+1)} = 0.
\end{align*}
In particular, $B - \bigcup_{v\in V} \Up(v)$ is non-empty and we fix any point $b$ in this set.

Now starting from the point $b$ we go down in the poset $P$.
Let $M_0=M^{(1/h)^{j-1}}$.
Note that by \ref{inv:C-vs-A-and-B} and $\calE \subset \calC$ we have $\calE^b = \calE$.
So using~\ref{inv:V-and-calC-sizes} we get $\norm{\calE^b} \geq M_0$.
Initially we set $v=b$, and as long as there is a point $x\in P$ such that $x$ is covered by $v$ in $P$ and 
\[
\norm{\calE^x} > \norm{\calE^v}/M_0^{1/(h-1)},\quad \textrm{we update $v=x$.}
\]

Note that the process must stop as the height of $v$ is decreasing in every move.
Furthermore, $v$ never goes down to a minimal point.
Indeed, if $x<v$ and $x$ is minimal in $P$, then at most $h-2$ steps were done and hence $\norm{\calE^v}>\norm{\calE^b}/M_0^{(h-2)/(h-1)}\geq M_0^{1/(h-1)}$.
On the other hand, $\norm{\calE^x} \leq 1=M_0^0$ as all sets in $\calE^x$ must contain $x$ when $x$ is minimal in $P$, and the sets in $\calE\subset\calC$ are pairwise disjoint (by~\ref{inv:C}).

Again, by invariant~\ref{inv:C} there is at most one set in $\calE$ containing $v$.
If such a set $C$ exists we define $\calE_*=\calE-\set{C}$, and otherwise we let $\calE_*=\calE$.

Now consider the set $X$ consisting of all points that are covered by $v$ in $P$.
As no set in $\calE_*$ contains $v$, we have $\calE_*^v = \bigcup_{x\in X} \calE_*^x$.
We want to ignore somewhat redundant covers of $v$, so take a minimal subset $X'$ of $X$ such that $\calE_*^v = \bigcup_{x\in X'} \calE_*^x$.
The minimality of $X'$ allows us to fix for every $x\in X'$ a set $C^x\in\calE_*^x -\bigcup_{y\in X'-\set{x}} \calE_*^y$.
Finally, we update our maintained pair to $(V',\calE')$, where
\[
 V'=V\cup\set{v}\quad \textrm{ and }\quad \calE' = \set{C^x\mid x\in X'}.
\]
See Figure~\ref{fig:fixing-metavertex} for an illustration of the set $\calE'$.
This finishes the $j$-th iteration of the loop.

\begin{figure}[t]
 \centering
 \includegraphics[scale=1.0]{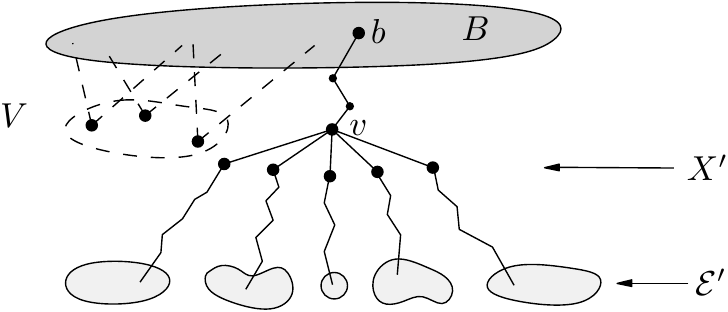}
 \caption{Point $b$ in $B-\Up(V)$ and a new point $v$ and its cover relations downwards in $P$.}
 \label{fig:fixing-metavertex}
\end{figure}

\begin{claim}
 The pair $(V',\calE')$ fulfills the invariants~\ref{inv:V-and-calC-sizes}-\ref{inv:clean-branching-for-V}.
\end{claim}
\begin{proof}
First of all, we show that $v\not\in V$ and therefore $\norm{V'}=j$.
Recall that we have chosen $b\in B$ such that $w\not\leq b$ in $P$, for every $w\in V$.
On the other hand, by our procedure we have $v \leq b$ in $P$ and hence $v\not\in V$.

Now we aim to get the required lower bound for $\norm{\calE'}$.
Since the sets $C^x$, $C^{y}$ are distinct for distinct $x,y\in X'$, we have $\norm{\calE'}=\norm{X'}$.
Moreover,
\begin{align*}
 \norm{\calE^v}&\leq \norm{\calE_*^v}+1=\norm{\bigcup_{x\in X'}\calE_*^x}+1\leq \sum_{x\in X'}\norm{\calE_*^x} +1\leq \sum_{x\in X'}\norm{\calE^x} +1\\
 &\leq \norm{X'}\cdot\norm{\calE^v}/M_0^{1/(h-1)}+1.
\end{align*}
Since $\norm{\calE^v}>M_0^{1/(h-1)}\geq 1$ and $M_0\geq 2^{h(h-1)}$ (this follows from~\ref{inv:V-and-calC-sizes} and $n\geq 3$), we deduce that
\[
 \norm{\calE'}=\norm{X'}\geq \frac{\norm{\calE^v}-1}{\norm{\calE^v}}\cdot M_0^{1/(h-1)}\geq\frac12 M_0^{1/(h-1)}\geq M_0^{1/h} = M^{(1/h)^{j}}.
\]
This proves~\ref{inv:V-and-calC-sizes}.

Invariant~\ref{inv:V-disjoint-from-calC} holds as $V$ is disjoint from every set in $\calE\supset\calE'$ and as $v$ is not contained in any set of $\calE_*\supset\calE'$.

It remains to verify~\ref{inv:clean-branching-for-V} for $(V',\calE')$.
Since $\calE'\subseteq \calE$, we only need to check this item for the new vertex $v$.
Consider a set $C\in\calE'$.
By the definition of $\calE'$ there is $x\in X'$ such that $C^x=C$, implying $\calE'^x=\set{C}$.
This completes the verification of the invariants for $(V',\calE')$.
\end{proof}

\bigskip

After the completion of Phase 2 we can finally construct a subdivision of $K_n$ in $\cover(P)$.
Let $(V,\calE)$ be the maintained pair after $n$ loop iterations.
By \ref{inv:V-and-calC-sizes} we have
\[
 \norm{V}=n \quad\textrm{ and }\quad \norm{\calE} \geq \binom{n}{2}.
\]
The points in $V$ will be the principal vertices of our subdivision of $K_n$.
We finish the construction leading to the edges of $K_n$.
Since $\norm{\calE}\geq \binom{n}{2}$, for every pair of distinct points $v_1,v_2\in V$ we can fix a unique set $C_{v_1v_2}\in\calE$.

By invariant~\ref{inv:clean-branching-for-V}, there are cover relations $x_1<v_1$ and $x_2<v_2$ in $P$ such that $\calE^{x_1}=\calE^{x_2}=\set{C_{v_1v_2}}$. 
In particular, there are $c_1,c_2\in C_{v_1v_2}$ such that $c_1\leq x_{1}< v_1$ and $c_2\leq x_{2}< v_2$ in $P$.
Let 
\[
 c_1=y_1<y_2<\cdots<y_{r}=x_{1}\ \text{ and }\ c_2=z_1<z_2<\cdots<z_{s}=x_{2}
\]
be chains consisting of cover relations in $P$.
Fix a path $P_{v_1v_2}$ connecting $v_1$ and $v_2$ in $\cover(P)$ using only vertices from the sets  $\set{y_1,\ldots,y_{r},v_1}$, $\set{z_1,\ldots,z_{s},v_2}$, and $C_{v_1v_2}$.
Such a path exists since $C_{v_1v_2}$ is connected in $\cover(P)$ (by~\ref{inv:C-connected}).

We claim that the union of these paths forms a subdivision of $K_n$ in $\cover(P)$.
All we need to prove is that whenever there is $z\in P_{v_1v_2}\cap P_{v_1'v_2'}$ for distinct two-sets $\set{v_1,v_2},\set{v_1',v_2'}\subset V$, 
then $z$ is an endpoint of both paths.
By the construction of our paths, there are cover relations $x_1 < v_1$, $x_2 < v_2$, $x_1' < v_1'$, and $x_2' < v_2'$ in $P$ with $\calE^{x_1}=\calE^{x_2}=\set{C_{v_1v_2}}$ and $\calE^{x_1'}=\calE^{x_2'}=\set{C_{v_1'v_2'}}$.

First, suppose that $z$ is an endpoint of one path and an internal point of the other path.
Without loss of generality we assume that $z=v_1$ (so it is an endoint of $P_{v_1v_2}$).
By the definition of $P_{v_1'v_2'}$, we have $z\leq x_1'$ in $P$, or $z\leq x_2'$ in $P$, or $z\in C_{v_1'v_2'}$.
In the first case it follows that $\calE=\calE^{v_1} \subseteq \calE^{x_1'} =\set{C_{v_1'v_2'}}$, which is a clear contradiction.
The second case is similar.
And the third one contradicts the fact that $V$ is disjoint from $\calE$ (by~\ref{inv:V-disjoint-from-calC}).

So suppose that $z\in P_{v_1v_2}\cap P_{v_1'v_2'}$ is an internal vertex of both paths.
The sets $C_{v_1v_2}$ and $C_{v_1'v_2'}$ cannot both contain $z$ as they are disjoint (by~\ref{inv:C}).
Hence we may assume $z\not\in C_{v_1v_2}$.
By the definition of $P_{v_1v_2}$, we then must have $z \leq x_1$ or $z\leq x_2$ in $P$.
Say $z\leq x_1$ holds in $P$.
Now observe that $z\in P_{v_1'v_2'}$ implies that there is $c'\in C_{v_1'v_2'}$ with $c'\leq z$ in $P$.
Hence $c'\leq x_{1}$ in $P$.
However, from this we deduce $C_{v_1'v_2'}\in \calE^{x_{1}}=\set{C_{v_1,v_2}}$, which is a contradiction.
We conclude that both path $P_{v_1v_2}$ and $P_{v_1'v_2'}$ are indeed internally disjoint.

As a consequence we established the existence of a subdivision of $K_n$ in $\cover(P)$.
This completes the proof of Theorem~\ref{thm:main-thm}.

\section{Proof of Theorem~\ref{thm:standard-ex}}
To prove Theorem~\ref{thm:standard-ex} we are going to use the framework from the previous section.
In particular, we run through two phases which will give us appropriate sets to construct a subdivision of $K_n$.
Compared to the previous section, the structure with its invariants in Phase 1 is slightly different.
For instance, the two collections will contain only singletons so that we can use two sets of points instead.
More importantly, we have new invariants~\ref{inv:k+k-C-distance} and \ref{inv:k+k-D-distance} that are somehow substituting the bounded height setting from the previous section.
Furthermore, the way we will get new elements (after a 'yes'-answer for the third question) is new.
Phase~2 and the construction of the subdivision will go along the same lines as in the first proof.

We omit the trivial cases and assume $n\geq3$.
Let $P$ be a $(k+k)$-free poset that contains a standard example $S_m$ with
\[
m > n^L, \textrm{ where } L=2\binom{M+k-1}{k-1}-1\textrm{ and }M=\binom{n}{2}^{(k-1)^n}. 
\]

During Phase~1 we maintain an additional structure $(A,B,C,D)$ while running a loop.
After the $i$-th iteration step we will have the following invariants:
\begin{enumeratea}
 \item there is a standard example of size $n^{L-i}$ in $P$ with $A$ and $B$ being the sets of its minimal and maximal points, respectively,\label{inv:k+k-A-B}
 \item $C$ is an antichain in $P$ with $\norm{C}\leq M$ and\label{inv:k+k-C}
 \begin{enumerateNested}
  \item $A \cap \D(c) =\emptyset$ and $B \subset \Up(c)$, for every $c\in C$,\label{inv:k+k-C-vs-A-and-B}
  \item all directed paths from $c$ to $b$ in $\cover(P)$ are of length less than $k$, for every $c\in C$, $b\in B$,\label{inv:k+k-C-distance}
 \end{enumerateNested}
 \item $D$ is an antichain in $P$ with $\norm{D}\leq M$ and\label{inv:k+k-D}
 \begin{enumerateNested}
  \item $B \cap \Up(d) =\emptyset$ and $A \subset \D(d)$, for every $d\in D$,
  \item all directed paths from $a$ to $d$ in $\cover(P)$ are of length less than $k$, for every $a\in A$, $d\in D$.\label{inv:k+k-D-distance}
 \end{enumerateNested}
\end{enumeratea}

Again, we have a measure of quality of the maintained structure.
For $c\in C$ let $\val(c)$ be the length of a longest directed path in $\cover(P)$ from $c$ to any $b\in B$.
Note that by invariant~\ref{inv:k+k-C-distance}, we have $\val(c)< k$ for every $c\in C$.
Then we define $\val(C)$ in the same way as in the previous section.
Analogously, we define $\val(D)$.
To compare the values we use the same relation as before.
Note that the number of possible values for $C$ and $D$, respectively, is bounded by $\binom{M+k-1}{k-1}$.

\subsection*{Phase 1: Updating the data structure}
We set up the initial structure as follows.
Fix a copy of a standard example of size $m$ in $P$. 
Set $A$ and $B$ to be the set of its minimal points and maximal points, respectively.
Set $C$ and $D$ to be the empty set. 
Clearly conditions \ref{inv:k+k-A-B}-\ref{inv:k+k-D} hold for $i=0$.

Now we run a loop to improve the quality of the data structure.
In each loop iteration we ask up to three questions about the current structure $(A,B,C,D)$.
If we get only negative answers, then the loop terminates and Phase~1 is done.
If we get a positive answer to one of the questions, then we finish the iteration by updating the structure to $(A',B',C',D')$ that is satisfying conditions \ref{inv:k+k-A-B}-\ref{inv:k+k-D} and additionally 
 \begin{align*}
\val(C')<\val(C) &\textrm{ and } \val(D')\leq\val(D), \textrm{ or }\\
\val(C')\leq\val(C) &\textrm{ and } \val(D')<\val(D).
\end{align*}
Since the number of possible values for the set $C$ (and $D$, resp.) is bounded by $\binom{M+k-1}{k-1}$, there will be at most $L=2\binom{M+k-1}{k-1}-1$ iterations in total.

Now we are going to describe the $i$-th ($1\leq i \leq L$) iteration in detail.
Let $(A,B,C,D)$ be the current structure.
Hence it satisfies conditions \ref{inv:k+k-A-B}-\ref{inv:k+k-D} after the $(i-1)$-th step.
The iteration starts with the evaluation of the following question:
\begin{align}\tag{Q1}\label{q:k+k-improving element}
\textrm{Is there a point $p\in P$ such that}
\end{align}
\begin{enumerate}
 \item $\norm{B \cap \Up(p)} > n^{L-i}$, and
 \item there is $c\in C$ such that $c< p$ in $P$?
\end{enumerate}

First, suppose that the answer is 'yes' and fix such a point $p\in P$.
In this case we finish the $i$-th iteration by updating the structure to $(A',B',C',D')$, where
\begin{align*}
B'&=B \cap \Up(p),\ A'=\set{a\in A \mid \textrm{$a$ is incomparable to some $b\in B'$}},\\
C'&=C\cup \set{p}- \set{c\in C\mid c < p \textrm{ in $P$}},\ \textrm{ and }\ D'=D.
\end{align*}
Note that $A'$ and $B'$ induce a standard example of size larger than $n^{L-i}$ and hence invariant~\ref{inv:k+k-A-B} is satisfied.
We skip the proof for the fact that $(A',B',C',D')$ satisfies the invariants \ref{inv:k+k-C}-\ref{inv:k+k-D} and moreover, that $\val(C')<\val(C)$ and $\val(D')\leq\val(D)$.
It follows along the same lines as in the argument for the analogue claim in the previous section.

If the answer for question~\eqref{q:k+k-improving element} is 'no', then the procedure continues with a dual question:
\begin{align}\tag{Q2}\label{q:k+k-dual improving element}
\textrm{Is there a point $p\in P$ such that}
\end{align}
\begin{enumerate}
 \item $\norm{A \cap \D(p)} > n^{L-i}$, and
 \item there is $d\in D$ such that $p< d$ in $P$?
\end{enumerate}
If the answer for~\eqref{q:k+k-dual improving element} is 'yes', then we finish the $i$-th iteration by updating the structure to $(A',B',C',D')$, where
\begin{align*}
A'&=A \cap \D(p),\ B'=\set{b\in B \mid \textrm{$b$ is incomparable to some $a\in A'$}},\\
C'&=C,\ \textrm{ and }\ D'=D\cup \set{p}-\set{d\in D\mid p < d \textrm{ in $P$}}.
\end{align*}
The proof that this new structure satisfies conditions \ref{inv:k+k-A-B}-\ref{inv:k+k-D} and that it improves the quality is dual to the one of question~\eqref{q:k+k-improving element}.

If the answers for questions~\eqref{q:k+k-improving element} and~\eqref{q:k+k-dual improving element} are both 'no', then the procedure continues with the third question:
\begin{align}\tag{Q3}\label{q:k+k-C or D full}
\textrm{Are $\norm{C}<M$ and $\norm{D}<M$?}
\end{align}

We first deal with the 'yes' answer.
In this case, we are going to show how to find a new element to extend $C$ or $D$.
This part of the procedure is simpler than its analogue in the previous section.
We do not unfold the poset with Lemma~\ref{lem:unrolling}, but instead make use of the structure of standard examples and the fact that $P$ is $(k+k)$-free.
(Actually, we do not know how to apply Lemma~\ref{lem:unrolling} to $(k+k)$-free posets, since the sets we obtain from it do not have a 'short-distance-property' like in invariant~\ref{inv:k+k-C-distance}.)

Fix an incomparable pair $(a_0,b_0)\in \Inc(A,B)$.
Consider the following partition of $A-\set{a_0}$ and $B-\set{b_0}$.
Let $A_{\textrm{far}}$ be the set of points $a\in A$ for which there exists a directed path from $a$ to $b_0$ of length at least $k$, and let $A_{\textrm{cl}}=A - A_{\textrm{far}}$.
Dually define $B_{\textrm{far}}$ and $B_{\textrm{cl}}$.

The key observation here is that for each incomparable pair $(a,b)\in\Inc(A,B)-\set{(a_0,b_0)}$,
we have $a\in A_{\textrm{cl}}$ or $b\in B_{\textrm{cl}}$.
Indeed, otherwise there are two directed paths in $\cover(P)$ of length at least $k$, one from $a_0$ to $b$ and one from $a$ to $b_0$.
And since $P$ is $(k+k)$-free, we would deduce that $a_0<b_0$ or $a<b$ in $P$, which is not true.
As a consequence, we get
\[
\norm{A_{\textrm{cl}}} \geq (\norm{A}-1)/2\ \textrm{ or } \norm{B_{\textrm{cl}}} \geq (\norm{B}-1)/2.
\]

Suppose first $\norm{B_{\textrm{cl}}} \geq (\norm{B}-1)/2$.
Then we finish the $i$-th iteration by updating the structure to $(A',B',C',D')$, where
\begin{align*}
 B'&=B_{\textrm{cl}},\ A'=\set{a\in A \mid \textrm{$a$ is incomparable to some $b\in B'$}},\\
 C'&=C\cup\set{a_0}, \textrm{ and } D'=D. 
\end{align*}
\begin{claim}
 The structure $(A',B',C',D')$ keeps the invariants~\ref{inv:A-B}-\ref{inv:D}.
 Moreover, $\val(C')<\val(C)$ and $\val(D')\leq\val(D)$.
\end{claim}
\begin{proof}
 Clearly, $A'$ and $B'$ form a standard example that is contained in the original standard example and its size is $\norm{A'}=\norm{B'}=\norm{B_{\textrm{cl}}}\geq (\norm{B}-1)/2 \geq (n^{L-(i-1)}-1)/2 \geq n^{L-i}$, as $n\geq 3$.
 Thus, \ref{inv:k+k-A-B} holds.
 
 Now we aim to prove that $C'=C\cup\set{a_0}$ is an antichain in $P$.
 Since $C$ is an antichain in $P$ (by~\ref{inv:C}), all we need to prove is that $a_0$ is incomparable to all points in $C$.
 To see this, note that $a_0 < c$ in $P$ for some $c\in C$ violates invariant~\ref{inv:k+k-C-vs-A-and-B} for $(A,B,C,D)$.
 And if $c < a_0$ in $P$ for some $c\in C$, then $a_0$ would be a witness for a 'yes' answer for question~\eqref{q:k+k-improving element}, 
 which is a contradiction.
 Hence $C'$ is an antichain in $P$ as required.
 
 Invariant~\ref{inv:k+k-C-vs-A-and-B} holds trivially as $a_0$ is incomparable to all points in $A'$ and below all points in $B'$.
 Invariant~\ref{inv:k+k-C-distance} holds as all directed paths from $a_0$ to any $b\in B'$ in $\cover(P)$ are of length less than $k$ by the definition of $B'=B_{\textrm{cl}}$.
 Since $D'=D$ and $A'\subset A$, $B'\subset B$, invariant~\ref{inv:k+k-D} is satisfied.
 
 Finally, $\val(C')<\val(C)$ and $\val(D')\leq\val(D)$ simply as $C\subsetneq C'$ and $D=D'$, respectively.
\end{proof}

If $\norm{A_{\textrm{cl}}} \geq (\norm{A}-1)/2$, then we finish the $i$-th iteration by updating the structure to $(A',B',C',D')$, where
\begin{align*}
 A'&=A_{\textrm{cl}},\ B'=\set{b\in B \mid \textrm{$b$ is incomparable to some $a\in A'$}},\\
 C'&=C, \textrm{ and } D'=D\cup\set{b_0}. 
\end{align*}
 The proof that $(A',B',C',D')$ keeps the invariants~\ref{inv:A-B}-\ref{inv:D} and that $\val(C')\leq\val(C)$ and $\val(D')<\val(D)$ goes dually to the previous case.

 \bigskip
 
Let $(A,B,C,D)$ be the final data structure of Phase 1.
From now on the proof goes exactly as the proof for Theorem~\ref{thm:main-thm}. 
So we proceed with Phase~2, which consists of a loop that has $n$ iterations.
To follow the proof from the previous section explicitly one should replace the sets $C$ and $D$ with the collections of their singletons.
Actually, since all the sets in the collections are singletons, the proof and invariant~\ref{inv:D} could be simplified.
The only important difference in the argument is that we previously used the fact that the poset has bounded height.
We replace it here with the observation that by~\ref{inv:k+k-C-distance} (and~\ref{inv:k+k-D-distance} analogously) the subposet induced by $\D(B)\cap \Up(C)$ has height less than $k$.

After Phase~2 we have a pair $(V,\calE)$ and we can construct a subdivision of $K_n$ exactly in the same way as in the proof of Theorem~\ref{thm:main-thm}.
This finishes the proof for Theorem~\ref{thm:standard-ex}.

\section*{Acknowledgements}
The authors send great thanks to Tom Trotter and Stefan Felsner for the time spent on discussions all over the topic, reading early versions of the manuscript, and sharing thoughts on improving the proof ideas and the exposition.

\bibliographystyle{plain}
\bibliography{posets-dimension}

\end{document}